\documentclass[final]{amsart}

\usepackage{xparse}
\usepackage{etoolbox}
\usepackage{kvoptions}

\usepackage{fix-cm}
\usepackage[T1]{fontenc}
\usepackage[utf8]{inputenc}
\usepackage[full]{textcomp}
\usepackage[l2tabu, orthodox]{nag}

\usepackage{lmodern} 
\usepackage[bitstream-charter]{mathdesign}

\usepackage{microtype}
\usepackage{csquotes}

\usepackage{graphicx}
\usepackage{tikz}
\usetikzlibrary{matrix, arrows, decorations.markings, decorations.pathmorphing, calc}

\tikzset{->-/.style={decoration={
  markings,
  mark=at position 0.5 with {\arrow{stealth}}},postaction={decorate}}}
\tikzset{->>-/.style={decoration={
  markings,
  mark=at position 0.5 with {\arrow{>>}}},postaction={decorate}}}
\tikzset{snake it/.style={decorate, decoration=snake}}

\usepackage{tikz-cd}
\usepackage{tikz-3dplot}
\usepackage{wrapfig}

\usepackage{caption}
\usepackage{subcaption}
\usepackage{multirow}
\usepackage{booktabs}
\usepackage{tabularx}
\usepackage{tabulary}
\usepackage{longtable}
\usepackage{siunitx}

\usepackage{dsfont}
\usepackage{xfrac}
\usepackage{mathtools}
\usepackage{mleftright}
\usepackage{manfnt}
\usepackage{accents}
\usepackage{faktor}

\DeclareFontFamily{U}{rsfs}{\skewchar\font127 }
\DeclareFontShape{U}{rsfs}{m}{n}{%
   <-6> rsfs5
   <6-8> rsfs7
   <8-> rsfs10
}{}

\usepackage{emptypage}
\usepackage{comment}

\usepackage[inline]{enumitem}
\usepackage[obeyFinal,colorinlistoftodos=true,textsize=footnotesize]{todonotes}

\usepackage[hyphens]{url}

\usepackage{hyperref}
\usepackage{geometry} 
\usepackage{amsthm}
\usepackage{thmtools,thm-restate}
\usepackage[capitalize]{cleveref}

\definecolor{dark-red}{rgb}{0.4,0.15,0.15}
\definecolor{dark-blue}{rgb}{0.15,0.15,0.4}
\definecolor{medium-blue}{rgb}{0,0,0.5}
\hypersetup{
    colorlinks, linkcolor={dark-red},
    citecolor={dark-blue}, urlcolor={medium-blue}
}

\LetLtxMacro{\amsmathdots}{\dots}
\usepackage{ellipsis}
\AtBeginDocument{\LetLtxMacro{\dots}{\amsmathdots}}

\DeclareMathOperator{\tr}{Tr}

\DeclareMathOperator{\Stab}{Stab}

\DeclareMathOperator{\GL}{GL}
\DeclareMathOperator{\PGL}{PGL}

\DeclareMathOperator{\Mat}{Mat}

\DeclareMathOperator{\PConf}{PConf}
\DeclareMathOperator{\UConf}{UConf}
\DeclareMathOperator{\Frob}{Frob}

\newcommand*{\et}{\text{ét}}

\newcommand*{\tolabel}[1]{\xrightarrow{#1}}

\newcommand*{\dispunct}[1]{\,\text{#1}}

\newcommand{\from}{\vcentcolon}

\newcommand{\dsum}{\oplus}
\newcommand{\dSum}{\bigoplus}
\newcommand{\tensor}{\otimes}

\NewDocumentCommand\xDeclarePairedDelimiter{mmm}
 {%
  \NewDocumentCommand#1{som}{%
   \IfNoValueTF{##2}
    {\IfBooleanTF{##1}{#2##3#3}{\mleft#2##3\mright#3}}
    {\mathopen{##2#2}##3\mathclose{##2#3}}%
  }%
 }

\xDeclarePairedDelimiter{\abs}{\lvert}{\rvert}
\xDeclarePairedDelimiter{\norm}{\lVert}{\rVert}
\xDeclarePairedDelimiter{\floor}{\lfloor}{\rfloor}
\xDeclarePairedDelimiter{\ceil}{\lceil}{\rceil}
\xDeclarePairedDelimiter{\gen}{\langle}{\rangle}
\xDeclarePairedDelimiter{\pseries}{\llbracket}{\rrbracket}

\NewDocumentCommand{\set}{somm}{%
   \IfNoValueTF{#2}
    {\IfBooleanTF{#1}{\{#3 \mid #4\}}{\mleft\{ #3 \mathrel{}\middle\vert\mathrel{} #4 \mright\}}}
    {\mathopen{#2\{}#3 \mathrel{}#2\vert\mathrel{} #4\mathclose{#2\}}}%
  }
\NewDocumentCommand{\present}{somm}{%
   \IfNoValueTF{#2}
    {\IfBooleanTF{#1}{\langle#3 \mid #4\rangle}{\mleft\langle#3 \mathrel{}\middle\vert\mathrel{} #4 \mright\rangle}}
    {\mathopen{#2\langle}#3 \mathrel{}#2\vert\mathrel{} #4\mathclose{#2\rangle}}%
  }
\NewDocumentCommand{\inner}{somm}{%
   \IfNoValueTF{#2}
    {\IfBooleanTF{#1}{\langle#3 , #4\rangle}{\mleft\langle#3 , #4 \mright\rangle}}
    {\mathopen{#2\langle}#3 , #4\mathclose{#2\rangle}}%
  }

\newcommand{\CC}{\mathbb{C}}
\newcommand{\FF}{\mathbb{F}}

\newcommand{\PP}{\mathbb{P}}
\newcommand{\QQ}{\mathbb{Q}}
\newcommand{\RR}{\mathbb{R}}

\newcommand{\cH}{\mathcal{H}}

\newcommand{\fS}{\mathfrak{S}}

\newcommand{\nor}{\text{nor}}

\newcolumntype{C}{>{\raggedright\arraybackslash}X}

\newcommand*{\widebar}[1]{\mkern 1.5mu\overline{\mkern-1.5mu#1\mkern-1.5mu}\mkern 1.5mu}

\newcommand*{\cl}[1]{
\begingroup
    \setbox\z@=\hbox{\ensuremath{#1}}%
    \ifdimgreater{\wd\z@}{4em}{\mleft(#1\mright)^{-}}{\widebar{#1}}
\endgroup
}

\newcommand*{\interior}[1]{
\begingroup
    \setbox\z@=\hbox{\ensuremath{#1}}%
    \ifdimgreater{\wd\z@}{1.5em}{\mleft(#1\mright)^{\circ}}{\accentset{\circ}{#1}}
\endgroup
}

\newcommand\isom{\xrightarrow{\,\smash{\raisebox{-0.6ex}{\ensuremath{\sim}}}\,}}

\numberwithin{equation}{section}
\declaretheorem[sibling=equation]{theorem}
\declaretheorem[sibling=theorem,style=remark]{example}
\declaretheorem[sibling=theorem,style=definition]{definition}
\declaretheorem[sibling=theorem]{lemma}

\declaretheorem[sibling=theorem]{corollary}
\declaretheorem[sibling=theorem]{proposition}

\declaretheorem[sibling=theorem, style=remark]{remark}
\declaretheorem[numbered=no, title=Theorem]{theorem*}
\declaretheorem[numbered=no, title=Corollary]{corollary*}

\usepackage{sseq}

\usepackage[style=ieee-alphabetic,url=false]{biblatex}

\renewbibmacro*{doi+eprint+url}{%
  \iftoggle{bbx:doi}
    {\iffieldundef{url}{\printfield{doi}}{}}
    {}%
  \newunit\newblock
  \iftoggle{bbx:eprint}
    {\usebibmacro{eprint}}
    {}%
  \newunit\newblock
  \iftoggle{bbx:url}
    {\usebibmacro{url+urldate}}
    {}}

\newcommand*{\Gr}{\mathrm{Gr}}

\newcommand*{\cV}{\mathcal{V}}

\newlist{singularity}{enumerate}{2}
\setlist[singularity,1]{label=(\Roman*),noitemsep, ref=\Roman*}
\setlist[singularity,2]{label=(\alph*),noitemsep, ref=\alph*}

\newcommand*{\type}[1]{\text{\ref{#1}}}

\makeatletter
\def\paragraph{\@startsection{paragraph}{4}%
  \z@\z@{-\fontdimen2\font}%
  {\normalfont\bfseries}}
\makeatother

\begin{document}

\title{The space of cubic surfaces equipped with a line}
\author{Ronno Das}
\address{Department of Mathematics, University of Chicago, Chicago, IL 60637, USA}
\email{ronno@math.uchicago.edu}

\begin{abstract}
The Cayley--Salmon theorem implies the existence of a $27$-sheeted covering space specifying lines contained in smooth cubic surfaces over $\CC$.
In this paper we compute the rational cohomology of the total space of this cover, using the spectral sequence in the method of simplicial resolution developed by Vassiliev.
The covering map is an isomorphism in cohomology (in fact of mixed Hodge structures) and the cohomology ring is isomorphic to that of $\PGL(4,\CC)$.
We derive as a consequence of our theorem that over the finite field $\FF_q$ the average number of lines on a cubic surface equals $1$ (away from finitely many characteristics); this average is $1 + O(q^{-1/2})$ by a standard application of the Weil conjectures.
\end{abstract}

\maketitle

\section{Introduction}

One of the first theorems of modern algebraic geometry and specifically enumerative geometry is the Cayley--Salmon theorem \cite{Cayley49}.
This classical theorem states that every smooth cubic surface (over an algebraically closed field, in particular $\CC$) contains exactly $27$ lines.
A cubic (hyper)surface in $\PP^3 = \CC P^3$ is the zero set $S = \cV(F)$ of a homogeneous polynomial $F$ of degree $3$ in $4$ variables.
The surface $S$ is singular (i.e.\ not smooth) if and only if the $20$ coefficients of $F$ are a zero of a discriminant polynomial $\Delta \from \CC^{20} \to \CC$.
Thus the space of smooth cubic surfaces is an open locus $M = M_{3,3} \vcentcolon = \PP^{19} \setminus \cV(\Delta)$.
The Cayley--Salmon theorem can be reinterpreted as a covering map $\pi \from \widetilde M \to M$, where $\widetilde M$ is the incidence variety of lines and smooth cubic surfaces (see \eqref{total-space-definition} and the preceding discussion for precise definitions).
The fiber $\pi^{-1}(S)$ over $S \in M$ is the set of $27$ lines on $S$.

The automorphism group of $\PP^3$ is $\PGL(4, \CC)$ and this group acts on lines and cubic surfaces, preserving smoothness.
In particular the covering map $\pi \from \widetilde M \to M$ is $\PGL(4, \CC)$-equivariant.
It was shown by Vassiliev (in \cite{Vassiliev99}) that the space $M$ has the same rational cohomology as $\PGL(4,\CC)$, and it follows from the results of Peters--Steenbrink (\cite{PS03}) that the orbit map given by $g \mapsto g(S_0)$ induces an isomorphism for any choice of $S_0 \in M$ (see~\cref{base-cohomology}).
See also \cite{Tommasi14}.

The main result of this paper is that the covering space $\widetilde M$ also has the same rational cohomology.
\begin{restatable}{theorem}{maintheorem}\label{main-theorem}
For a choice $(S_0,L_0) \in \widetilde M$, the orbit map $\PGL(4,\CC) \to \widetilde M$ given by $g \mapsto g(S_0,L_0)$ induces an isomorphism 
\[H^*(\widetilde M ; \QQ) \isom H^*(\PGL(4,\CC); \QQ) \cong \QQ[a_3,a_5,a_7]/(a_3^2,a_5^2,a_7^2) \dispunct,\]
where $a_i \in H^i(\PGL(4,\CC); \QQ)$. 
Since the composition $\PGL(4, \CC) \to \widetilde M \tolabel{\pi} M$ also induces an isomorphism on $H^*(\_; \QQ)$, the map 
\[\pi^* \from H^*(M; \QQ) \to H^*(\widetilde M; \QQ)\]
is an isomorphism.
Since the orbit map and $\pi$ are algebraic, the isomorphisms are of mixed Hodge structures.
\end{restatable}

\begin{remark}
In particular, $H^k(\widetilde M; \QQ)$ is pure of Tate type; the generator $a_{2k-1}$ is of bidegree $(k,k)$.
\end{remark}

The main tool in our proof of \cref{main-theorem} is simplicial resolution à la Vassiliev.
However the introduction of a line significantly increases the combinatorics of the casework.
We devote all of \cref{case-work} to this computation, while \cref{proof-section} contains the rest of the proof.

\subsection{Applications: moduli space, representations of \texorpdfstring{$W(E_6)$}{W(E\_6)} and point counts}\label{applications}

Before presenting a proof of \cref{main-theorem}, which we postpone to \cref{main-theorem-proof} and the particularly tedious details further to \cref{case-work}, we describe a few applications.
All of the corollaries in this section are corollaries to \cref{main-theorem}.

\subsubsection*{Cohomology of moduli spaces}
The map $\pi \from \widetilde M \to M$ is $\PGL(4,\CC)$ equivariant and each orbit (in either $M$ or $\widetilde M$) is closed (see~e.g.~\cite{ACT02}). 
Thus passing to the geometric quotient we get a covering map
\[\cH_{3,3}(1) \to \cH_{3,3} \dispunct,\]
where 
\[\cH_{3,3} = M/\PGL(4, \CC)\]
is the moduli space of smooth cubic surfaces and 
\[\cH_{3,3}(1) = \widetilde M/\PGL(4, \CC)\]
is the moduli space of cubic surfaces equipped with a line.
Note that both $\cH_{3,3}$ and $\cH_{3,3}(1)$ are coarse moduli spaces.
For example the Fermat cubic defined by $x^3 + y^3 + z^3 + w^3$ equipped with the line $\{x = y, z = w\}$ has non-trivial (but finite) stabilizer in $\PGL(4, \CC)$.
Using \cite[Theorem 2]{PS03}, which is a generalization of the Leray--Hirsch theorem, we have the following corollary.
\begin{corollary}\label{moduli-space}
The space $\cH_{3,3}(1)$ is $\QQ$-acyclic: $H^i(\cH_{3,3}(1); \QQ) = 0$ for $i > 1$.
\end{corollary}
For comparison, it was already known by \cref{base-cohomology} that $\cH_{3,3}$ is $\QQ$-acyclic.
Various compactifications of $\cH_{3,3}$, $\cH_{3,3}(1)$ and other covers can be found in \cite{DGK05}, in particular the two moduli spaces mentioned here are rational.
Also relevant are the computation of $\pi_1(\cH_{3,3})$ (as an orbifold) by Looijenga \cite{Looijenga08}, the identification of a compactification of $\cH_{3,3}$ as a quotient of complex hyperbolic $4$-space by Allcock, Carlson and Toledo \cite{ACT02}.

\subsubsection*{The cohomology of the normal cover as a representation of \texorpdfstring{$W(E_6)$}{W(E\_6)}}
The combinatorics of how the $27$ lines intersect is extremely well-studied.
Let $L$ be the graph with vertices the $27$ lines and edges corresponding to intersecting pairs for the generic cubic surface \parencite{Cayley49}.
It was classically known that the automorphism group of $L$ is realized as the Galois group of the extension given by adjoining the coefficients defining the lines over the field containing the coefficients of a cubic form.
Camille Jordan proved \cite{Jordan89} that this group is the Weyl group $W(E_6)$ of the root system $E_6$ (see~also~\cite[Remark 23.8.2]{Manin86}).
The Galois group can also be realized as the monodromy of the covering space $\widetilde M \to M$ and hence the deck group of its normal closure; see~\cite{Harris79}.

The cover $\widetilde M \to M$ is in fact not normal (Galois): its normal closure is the space $\widetilde M_{\nor}$ consisting of pairs $(S,\alpha)$, where $\alpha$ is an identification of the intersection graph of the $27$ lines on $S$ with $L$.
The deck group of $\widetilde M_{\nor}$ is $W(E_6)$, as mentioned, and so $H^*(\widetilde M_{\nor}; \QQ)$ is a $W(E_6)$ representation.
We can restrict this representation to the index-$27$ subgroup that stabilizes a line, which can be identified with $W(D_5)$ (see~\cite{Naruki82}).
The intermediate cover corresponding to this $W(D_5)$ is exactly $\widetilde M$. 
We can now deduce the following corollary about $H^*(\widetilde M_{\nor}; \QQ)$ from \cref{main-theorem}.
\begin{corollary}\label{normal-cover}
For any non-trivial irreducible representation $V$ of $W(E_6)$ appearing in $H^*(\widetilde M_{\nor}; \QQ)$, the restriction of $V$ to $W(D_5)$ cannot have a trivial summand.
Equivalently, the non-trivial irreducible representations of $W(E_6)$ that occur in the $27$-dimensional permutation representation given by the action on left cosets of $W(D_5)$ in $W(E_6)$ cannot occur in $H^*(\widetilde M_{\nor}; \QQ)$.
\end{corollary}
\begin{proof}
By \cref{main-theorem} and transfer, \[
H^*(\widetilde M_{\nor}; \QQ)^{W(E_6)} = H^*(M; \QQ) = H^*(\widetilde M; \QQ) = H^*(\widetilde M_{\nor}; \QQ)^{W(D_5)} \dispunct.\]
The second statement is equivalent to the first by Frobenius reciprocity.
\end{proof}
Computing the cohomology $H^*(\widetilde M_{\nor}; \QQ)$ (as a $W(E_6)$ representation) would be an obvious and major generalization of \cref{main-theorem}.
While the above corollary provides a restriction towards which irreducible representations can occur, it only rules out a small fraction: the order of $W(E_6)$ is $51840$, and it has $24$ non-trivial irreducible representations (see~\cite[428--429]{Carter85} for a character table).

There are other intermediate covers of $M$, by marking different configurations of the $27$ lines.
For instance, taking unordered triples of lines that intersect pairwise, we get a $45$-sheeted cover marking the `tritangents' of a cubic surface.
See \cite{Naruki82} and the appendix by Looijenga for more on this cover and its quotient under $\PGL(4,\CC)$.

\subsubsection*{Lines over $\FF_q$}
The spaces $\widetilde M$ and $M$ as defined above are (the complex points of) quasiprojective varieties defined by integer polynomials. 
To be more explicit, the discriminant $\Delta$ is an integer polynomial, as are the polynomials defining the incidence of a line and a cubic surface.
For a finite field $\FF_q$ of characteristic $p$, we can base change to $\FF_q$.
That is, reducing the defining polynomials $\operatorname{mod} p$ defines spaces 
\[M(\FF_q) \subset \PP^{19}(\FF_q) \dispunct,\]
\[\widetilde M(\FF_q) \subset \PP^{19}(\FF_q) \times \Gr(2,4)(\FF_q) \dispunct,\]
and a projection map
\[\pi \from \widetilde M(\FF_q) \to M(\FF_q) \dispunct.\]

For $p \ne 3$, the discriminant $\Delta$ continues to characterize singular polynomials, so $M(\FF_q)$ is the space of smooth cubic surfaces defined over $\FF_q$ (where a homogeneous cubic polynomial is smooth if it is smooth at all $\widebar{\FF_q}$ points).
Similarly, $\widetilde M(\FF_q)$ is the space of pairs $(S,L)$ of smooth cubic surfaces $S$ and lines $L$ defined over $\FF_q$ such that $L \subset S$.
Thus, $\dfrac{\# \widetilde M(\FF_q)}{\# M(\FF_q)}$ is the average number of $\FF_q$-lines on a cubic surface defined over $\FF_q$.
The Grothendieck--Lefschetz fixed point formula (see~e.g.~\cite{Milne13}) lets us use our results to deduce consequences about the cardinality of $\#\widetilde M(\FF_q)$.

\begin{remark}
The fact that $\widetilde M$ is a connected cover of $M$ already implies $H^0(\widetilde M; \QQ) \cong \QQ$.
Given Deligne's theorem \cite[Théorème 3.3.1]{Deligne80} we get that both $\#M(\FF_q)$ and $\# \widetilde M(\FF_q)$ are $q^{19}(1 + O(q^{-1/2}))$, since $\dim M = \dim \widetilde M = 19$.
Hence the average number of lines on a $\FF_q$-cubic surface is $1 + O(q^{-1/2})$ as $q \to \infty$.
One needs much more information to compute this number exactly.
\end{remark}

\begin{corollary}\label{average-lines}
There is a finite set of characteristics, so that for a fixed $q$ with $p$ not in this set, 
\[\# M(\FF_q) = \# \widetilde M(\FF_q) = q^4 (\# \PGL(4, \FF_q)) = q^4\frac{(q^4-1)(q^4-q)(q^4-q^2)(q^4-q)}{q-1}\dispunct.\]
Thus the average number of lines defined over $\FF_q$ on a smooth cubic surface defined over $\FF_q$ is exactly $1$.
\end{corollary}

To the best of our knowledge, the point count for $\widetilde M(\FF_q)$ and the consequence about the average number of lines is new.

\begin{proof}[Proof of \cref{average-lines}]
The varieties $M$ and $\widetilde M$ are smooth since $M$ is open in $\PP^{19}$.
For a smooth quasiprojective variety $Y$, the $\FF_q$ points are exactly the fixed points of $\Frob_q$ on $Y(\widebar{\FF_q})$, and $\#Y(\FF_q)$ is determined by the Grothendieck--Lefschetz fixed point formula (see~e.g.~\cite{Milne13}):
\[\# Y(\FF_q) = q^{\dim Y}\sum_{i \ge 0} (-1)^i \tr(\Frob_q \from H^i_{\et}(Y; \QQ_\ell)^\vee) \dispunct,\]
where $\ell$ is a prime other than $p$.
Further, there are comparison theorems implying isomorphisms
\[H^i_\et(Y; \QQ_\ell) \cong H^i(Y(\CC); \QQ_\ell) \cong H^i(Y(\CC); \QQ) \tensor \QQ_\ell \dispunct,\]
away from a finite set of characteristics (see~e.g.~\cite[Théorème 1.4.6.3, Théorème 7.1.9]{Deligne77}).
In particular, as a corollary of \cref{main-theorem} we obtain $\# \widetilde M(\FF_q) = \# M(\FF_q) = q^4(\#\PGL(4,\FF_q))$ and hence the corollary.
\end{proof}

%

\begin{remark}
One can define $\cH_{3,3}(\FF_q)$ and $\cH_{3,3}(1)(\FF_q)$ as base-changes of $\cH_{3,3}$ and $\cH_{3,3}(1)$ from above.
Using an analogue of the Groethendieck--Lefschetz fixed-point formula, it is possible to conclude that
\[\# \cH_{3,3}(1)(\FF_q) = \# \cH_{3,3}(\FF_q) = q^4 \dispunct,\]
although one needs to be more careful in interpreting these `point counts' mean.
However, a deeper discussion of the arguments involved is out of the scope of this paper.
\end{remark}

\subsection{Acknowledgments}
I would like to thank Benson Farb for his invaluable advice and comments throughout the composition of this paper and also for suggesting the problem.
I am also grateful to Weiyan Chen, Nir Gadish, Sean Howe and Akhil Mathew for many helpful conversations.
I am grateful to Igor Dolgachev for pointing me towards some existing results about moduli spaces of cubic surfaces.
Finally, I would like to thank Maxime Bergeron, Priyavrat Deshpande, Eduard Looijenga and Jesse Wolfson for their helpful comments to make the paper more readable.

\section{Rational cohomology of the incidence variety}\label{main-section}

\subsection{Definitions and setup}
From now on we will work over the field $\CC$ of complex numbers.
Let $X = X_{3,3}$ be the space of \emph{smooth} homogeneous degree $3$ (complex) polynomials over $4$ variables, for concreteness a subset of $\CC[x,y,z,w]_3 \cong \CC^{20}$.
A polynomial $F \in \CC[x,y,z,w]_3$ is smooth precisely when $\{F_x,F_y,F_z,F_w\}$ do not have a common root, by Euler's formula.
This is equivalent to a certain `discriminant' in the coefficients not vanishing; there is a polynomial $\Delta \from \CC^{20} \to \CC$ with integer coefficients that vanishes on (the coefficients of) $F$ if and only if $F$ is not smooth.
In other words, $X$ is the complement of the \emph{discriminant locus}, $\Sigma = \cV(\Delta) \subset \CC^{20}$.

We also have the `incidence variety' of a line and a (not necessarily smooth) cubic polynomial
\[\Pi = \set{(F,L)}{F|_L \equiv 0} \subset \CC[x,y,z,w]_3 \times \Gr(2,4) \dispunct,\]
where $\Gr(2,4)$ is the Grassmannian of lines in $\PP^3$ (that is, $2$-planes in $\CC^4$).
This space comes equipped with two projections.
The first, $\pi \from (F,L) \mapsto F$ forgets the line, and we denote the inverse image $\pi^{-1}(X)$ of $X$ by $\widetilde X$, which by (a version of) the Cayley--Salmon theorem is a $27$-sheeted cover $\pi \from \widetilde X \to X$.

The second projection is to $\Gr(2,4)$, given by $(F,L) \mapsto L$, and is a fiber bundle with fiber $\Pi_\ell \cong \CC^{16}$ over $\ell \in \Gr(2,4)$.
To be explicit, $\Pi_\ell$ is the space of (not necessarily smooth) cubic polynomials that vanish on $\ell$.
The restriction of the projection to $\widetilde X$ is also a fiber bundle, and we will denote the fiber over $\ell$ by $\widetilde X_{\ell}$, this is the space of smooth homogeneous cubic polynomials in $4$ variables that vanish on $\ell$.
Let 
\[\Sigma_\ell = \Pi_\ell \setminus \widetilde X_\ell = \Sigma \cap \Pi_\ell \dispunct.\]

To go from the space of polynomials to the space of cubic surfaces, we need to quotient by the action of $\CC^\times$.
Namely, given a homogeneous cubic polynomial $F$ and $\lambda \in \CC^\times$, the product $\lambda F$ is another homogeneous cubic polynomial which defines the same surface $\cV(F) = \cV(\lambda F)$ and $F$ is smooth if and only if $\lambda F$ is.
Alternatively viewed, $\Delta$ is a homogeneous polynomial and $\Sigma$ is a conical hypersurface in $\CC^{20}$, so passing to the quotient by $\CC^\times$ produces spaces 
\[M = X_{3,3}/\CC^\times \subset \PP^{19} \dispunct,\] 
\begin{equation}\widetilde M = \widetilde X/\CC^\times \subset M \times \Gr(2,4) \label{total-space-definition}\end{equation}
and a covering map $\widetilde M \to M$, which we will also denote by $\pi$.

The map $\widetilde M \to \Gr(2,4)$ continues to be a fiber bundle, we denote the fiber over $\ell \in \Gr(2,4)$ by 
\[\widetilde M_\ell = \set{(S,\ell)}{S \in M, S \supset \ell}\dispunct.\]
All these spaces and the maps described so far fit into the following (somewhat clumsy) commuting diagram:
\begin{equation}\label{bundle-map-on-grassmannian}
\begin{tikzcd}
\widetilde X_\ell \arrow[rd, hook] \arrow[rrr, "\CC^\times"] & & & \widetilde M_\ell \arrow[rd, hook] &  \\
 & \widetilde X \arrow[ld, "27"] \arrow[rrr, "\CC^\times"] & & & \widetilde M \arrow[ld, "27"] \\
X \arrow[rrr, near end, "\CC^\times"] & & & M &  \\
 & \Gr(2,4) \arrow[rrr, equal] \arrow[from=uu, crossing over] & & & \Gr(2,4) \arrow[from=uu, crossing over]
\end{tikzcd}
\end{equation}

There is one more action to consider, which is important for both our theorem and its proof.
As mentioned in the introduction, $\GL(4) \vcentcolon = \GL(4,\CC)$ acts on $\CC^4$ and $\PGL(4) = \GL(4)/(\CC^\times I)$ acts on the quotient $\PP^3$.
There are induced actions on the spaces defined above: on $X$ and $\widetilde X$ by $\GL(4)$; on $M$ and $\widetilde M$ by $\PGL(4)$.
The action of $\GL(4)$ on $\Gr(2,4)$ also factors through $\PGL(4)$.
Fixing a line $\ell \in \Gr(2,4)$, the respective stabilizers in $\GL(4)$ and $\PGL(4)$ act on the fibers $\widetilde X_\ell$ and $\widetilde M_\ell$.
If we fix a basepoint $(F_0,L_0) \in \widetilde X$, and set $S_0 = \cV(F_0)$ so that $(S_0,L_0) \in \widetilde M$, we get orbit maps $g \mapsto g(S_0,L_0) = (g\cdot S_0, g \cdot L_0)$, and so on.
Then we also have the following commuting diagram:
\begin{equation}\label{bundle-map-orbit}
\begin{tikzcd}
\CC^\times \arrow[rd, hook] \arrow[r,"z \mapsto z^3"] & \CC^\times \arrow[rd, hook] \arrow[r] & \CC^\times \arrow[rd, hook]\\ 
& \GL(4) \arrow[r] \arrow[d] & \widetilde X  \arrow[d] \arrow[r,"\pi"] & X \arrow[d] \\
& \PGL(4) \arrow[r] & \widetilde M  \arrow[r,"\pi"] & M
\end{tikzcd}
\end{equation}
All the four maps in the bottom-left square are in fact maps of bundles over the same base $\Gr(2,4)$, and all the vertical maps are bundles with fiber $\CC^\times$.
The second and third vertical maps are elaborated in the previous diagram (\ref{bundle-map-on-grassmannian}).

\begin{remark}
It is worth noting that the map on the fibers $\CC^\times \to \CC^\times$ induced by the first horizontal map is not identity, the matrix $\omega I$ acts by $\omega^3 = 1$ on a cubic polynomial $F$.
As indicated, it is the degree $3$ map $z \mapsto z^3$, which is an isomorphism with rational coefficients, so this does not affect our computations.
\end{remark}

\begin{remark}
Since $\widetilde M$ is connected, the orbit maps for different choices of basepoint $(S_0, L_0) \in \widetilde M$ are homotopic.
\end{remark}

As mentioned in the introduction, Vassiliev's results imply that $M$ and $\PGL(4)$ have the same rational cohomology.
\begin{theorem}[Vassiliev \cite{Vassiliev99}, Peters--Steenbrink \cite{PS03}] \label{base-cohomology}
The map $\PGL(4) \to M$ given by $g \mapsto g(S_0)$ induces an isomorphism 
\[H^*(M; \QQ) \isom H^*(\PGL(4); \QQ) \dispunct.\]
\end{theorem}

By transfer we know that $\pi^* \from H^*(M; \QQ) \to H^*(\widetilde M; \QQ)$ is an injection.
This also follows from the fact that the orbit map in the above theorem factors through $\widetilde M$.
In fact, there is no new cohomology that appears in this cover, as in \cref{main-theorem}.

\subsection{Proof of \cref*{main-theorem} and the role of simplicial resolution}
\label{proof-section}
Vassiliev's method of simplicial resolution works by first reducing the computation of the cohomology of the discriminant complement $X$ to computing the (Borel--Moore) homology of the discriminant locus $\Sigma$ via Alexander duality.
The space $\Sigma$ consisting of the singular cubic surfaces is itself highly singular, and stratifies based on the how big the singular set of each $F \in \Sigma$ is.
Applying the spectral sequence of a filtration to this stratification produces a spectral sequence converging to $\widebar H_*(\Sigma) = H_*^{\mathrm{BM}}(\Sigma)$ (Borel--Moore or compactly supported homology).

While $\widetilde M$ or $\widetilde X$ is not an open subset of a vector space, recall that the fiber $\widetilde X_\ell$ over $\ell$ of the map $\widetilde X \to \Gr(2,4)$ is open in the vector space $\Pi_\ell$ of polynomials vanishing on $\ell$.
So we can apply the Vassiliev spectral sequence to each $\widetilde X_\ell$ to find $H^*(\widetilde X_\ell; \QQ)$.
For this, we need to stratify $\Sigma_\ell = \Sigma \cap \Pi_\ell$ by not just how big the singular sets are, but how they are configured with respect to the line $\ell$.
These are the types and subtypes described in \cref{plan}.
For now we will assume that we can perform this computation (which takes up all of \cref{case-work}), and when needed we refer to the answer described in \cref{cohomology-of-X-l}.

\begin{lemma}\label{hypersurface-complement}
Let $f \from \CC^n \to \CC$ be a non-constant homogeneous polynomial of degree $d$, so that $\cV(f)$ is a conical hypersurface; let its complement be $Y = \CC^n \setminus \cV(f)$.
Let $\PP Y = Y /\CC^\times = \PP^{n-1} \setminus \cV_\PP(f)$ be the complement of the projective hypersurface given by the same polynomial $f$.
Then $H^*(Y; \QQ) \cong H^*(\CC^\times) \tensor H^*(\PP Y)$.
\end{lemma}
\begin{proof}
We have a fiber bundle
\[\begin{tikzcd}
\CC^\times \rar[hook] & Y \dar[->>]\\
& \PP Y
\end{tikzcd}\]
and for any fiber $\CC^\times a$, the map $\CC^\times a \hookrightarrow Y \tolabel{f} \CC^\times$ is given by $\lambda \mapsto \lambda^d f(a)$, which is degree $d\ne 0$ on $\CC^\times$ and hence an isomorphism on $H^*(\CC^\times; \QQ)$.
This implies the conclusion by the Leray--Hirsch theorem.
\end{proof}

\begin{lemma}\label{surjection}
For a fixed $\ell \in \Gr(2,4)$, let $\Stab_{\GL(4)}(\ell)$ be the stabilizer of $\ell$ in $\GL(4)$.
Then for a choice of basepoint $F_0 \in \widetilde X_\ell$, the orbit map $\Stab(\ell) \to \widetilde X_\ell$ given by $g \mapsto g(F_0) = F_0 \circ g$ induces a surjection
\[H^*(\widetilde X_\ell; \QQ) \twoheadrightarrow H^*(\Stab_{\GL(4)}(\ell); \QQ) \cong H^*(\GL(2) \times \GL(2); \QQ) \dispunct.\]
\end{lemma}

\begin{proof}
First, fix a complement $\ell^\perp$ of $\ell$ (as the notation suggests, we can pick the orthogonal complement of $\ell$).
Then $\Stab_{\GL(4)}(\ell)$ deformation retracts to $G = \Stab_{\GL(4)}(\ell, \ell^\perp)$ (the elements that fix both $\ell$ and $\ell^\perp$).
Further, 
\[G = \GL(\ell) \times \GL(\ell^\perp) \dispunct.\]

As in the computation of $H^*(\widetilde X_\ell; \QQ)$ in \cref{case-work}, it is important to identify via Alexander duality $H^*(\widetilde X_\ell; \QQ)$ with $\widebar H_*(\Sigma_\ell)$, and similarly $H^*(\GL(2); \QQ)$ with $\widebar H_*(\Mat(2) \setminus \GL(2))$, where $\Mat(2)$ is the space of all $2 \times 2$ matrices.
The generators of $H^*(\GL(2); \QQ)$ (as a ring) are represented by the locus of matrices whose first $i$ columns are linearly dependent\footnote{For $i = 1$ this means the first column is $0$.
This description of the generators generalizes to $\GL(n) \subset M(n)$.}, for $i = 1, 2$.

Fix $P \in \ell$ and $P' \in \ell^\perp$ non-zero and extend to bases of $\ell$ and $\ell^\perp$ respectively.
This identifies $\GL(\ell) \times \GL(\ell^\perp) \cong \GL(2) \times \GL(2)$.
The orbit map extends to a map 
\[\Mat(2) \times \Mat(2) \to \Pi_\ell = \widetilde X_\ell \cup \Sigma_\ell \dispunct.\]
It is enough to identify subspaces of $\Sigma_\ell$ that pull-back to (a rational multiple of) the corresponding subspaces of $\Mat(2) \times \Mat(2)$.
Then directly from arguments in \cite[section 6]{PS03}, it is enough to pick the following four subspaces of polynomials that are: (i) singular at $P$, (ii) singular at some (non-zero) point of $\ell$, (iii) singular at $P'$, (iv) singular at some (non-zero) point of $\ell^\perp$.
\end{proof}

Now we prove the analogue of \cref{main-theorem} before projectivization.
\begin{proposition}
The orbit map $\GL(4) \to \widetilde X$ and the projection $\pi \from \widetilde X \to X$ induce isomorphisms
\[H^*(X; \QQ) \isom H^*(\widetilde X; \QQ) \isom H^*(\GL(4); \QQ) \dispunct.\]
\end{proposition}

\begin{proof}
Note that by \cref{cohomology-of-X-l},
\[H^*(\widetilde X_\ell; \QQ) \cong H^*( \GL(2) \times \GL(2); \QQ) \cong H^*(\Stab_{\GL(4)}(\ell); \QQ) \dispunct.\]
Since the orbit map $\Stab_{\GL(4)}(\ell) \to \widetilde X_\ell$ induces a surjection on $H^*(\_; \QQ)$ by \cref{surjection}, the induced map must be an isomorphism.

Thus we have a map of bundles (as in \eqref{bundle-map-on-grassmannian})
\[\begin{tikzcd}
\Stab_{\GL(4)}(\ell) \arrow[rd, hook] \arrow[r] &  \widetilde X_\ell \arrow[rd, hook] &  \\
 & \GL(4) \arrow[dd] \arrow[r] &  \widetilde X \arrow[dd] \\ \\
 & \Gr(2,4) \arrow[r, equal] & \Gr(2,4) 
\end{tikzcd}\]
that fiberwise induces an isomorphism 
\[H^*(\widetilde X_\ell; \QQ) \isom H^*(\Stab_{\GL(4)}(\ell); \QQ) \dispunct.\]
There is no monodromy in either bundle since $\Gr(2,4)$ is simply connected.
Therefore from naturality of the Serre spectral sequence, the map $\GL(4) \to \widetilde X$ must also be an isomorphism on cohomology.
\end{proof}

Converting this to a proof of \cref{main-theorem} is fairly simple.
We restate the theorem here for convenience.
\maintheorem*
\begin{proof}[Proof of \cref{main-theorem}]\label{main-theorem-proof}
We have another map of bundles (as in \eqref{bundle-map-orbit}):
\[\begin{tikzcd}
\CC^\times \arrow[rd, hook] \arrow[r,"z \mapsto z^3"] & \CC^\times \arrow[rd, hook] \\ 
& \GL(4) \arrow[r] \arrow[d] & \widetilde X  \arrow[d] \\
& \PGL(4) \arrow[r] & \widetilde M 
\end{tikzcd}\]
By \cref{hypersurface-complement}, both of these bundles satisfy the Leray--Hirsch theorem and the fiberwise map $\CC^\times \to \CC^\times$ is degree $3$, so induces an isomorphism on $H^*(\CC^\times; \QQ)$.
Thus the map of bases $\PGL(4) \to \widetilde M$ must also induce an isomorphism 
\[H^*(\widetilde M; \QQ) \isom H^*(\PGL(4); \QQ) \dispunct. \qedhere\]
\end{proof}

\section{Rational cohomology of \texorpdfstring{$\widetilde X_{\ell}$}{X-tilde(l)}}
\label{case-work}


\subsection{Definitions and plan of attack}\label{plan}

We will suppress constant rational coefficients throughout this section, and use $\widebar H$ to denote Borel--Moore homology (also with rational coefficients by default).
Note that for an orientable but not necessarily compact $2n$-manifold $M$, Poincaré duality takes the form
\[\widebar H_i(M) \cong H^{2n-i}(M) \cong (H_{2n-i}(M))^\vee \cong (H^{i}_c(M))^\vee \dispunct.\]

We use the spectral sequence developed by Vassiliev in \cite{Vassiliev99}.
We refer the reader to Vassiliev's paper for the theory, but summarize how the computation works in practice.
Recall that $\widetilde X_\ell \subset \Pi_\ell \cong \CC^{16}$, and set $\Sigma_\ell = \Pi_\ell \setminus \widetilde X_\ell = \Pi_\ell \cap \Sigma$, the set of singular cubic polynomials that vanish on the line $\ell$.
Then via Alexander duality, 
\begin{equation}\label{alexander-duality}
\widetilde H^i(\widetilde X_\ell) = \widebar H_{31 - i}(\Sigma_\ell) \dispunct.
\end{equation}
Note that $\Sigma_\ell$ is a hypersurface in $\Pi_\ell$, being the vanishing locus of $\Delta_\ell = \Delta|_{\Pi_\ell}$.

\begin{remark}\label{andreotti--frankel}
The complex variety $\widetilde X_\ell$, being the complement of a hypersurface, is affine and hence a $16$-dimensional Stein manifold.
Thus by the Andreotti--Frankel theorem, $H^i(\widetilde X_\ell) = 0$ for $i > 16$.
This along with \cref{alexander-duality} imply that $\widebar H_{i}(\Sigma_\ell)$ can only be non-zero for $15 \le i \le 31$.
\end{remark}

Let $F \in \Sigma_\ell$ be a singular cubic polynomial and let $K$ be its singular locus.
Then $K$, as a subset of $\PP^3$, can be one of the following $11$ types (see~\cite[Proposition 8]{Vassiliev99}):
\begin{singularity}
\item a point; \label{point}
\item two distinct points; \label{two-points}
\item a line;
\item three points, not on a line; \label{three-points}
\item a smooth conic contained in a plane $\PP^2 \subset \PP^3$;
\item a pair of intersecting lines;
\item four points, not on a plane;\label{four-points}
\item a plane;
\item three lines through a point, not all on the same plane;
\item a smooth conic contained in a plane along with another point not on that plane;
\item all of $\PP^3$ \label{everything}.
\end{singularity}

These further break up as \emph{subtypes} depending on their configuration with respect to $\ell$.
For most of the types, how they break up will not be relevant to us; we list those that will.
We list names for the points for convenience, they are still to be thought of as a priori \emph{unordered} sets of points: $\{P,Q\} = \{Q,P\}$ and so on.
\begin{singularity}
\item[(\ref{point})] a point $P$
\begin{singularity}[ref=\ref*{point}\alph*]
\item $P \in \ell$\label{one-on}
\item $P \notin \ell$\label{one-off}
\end{singularity}
\item[(\ref{two-points})] two points $P$, $Q$
\begin{singularity}[ref=\ref*{two-points}\alph*]
\item $P, Q \in \ell$\label{two-on}
\item $P \in \ell$, $Q \notin \ell$ \label{one-on-one-off}
\item $P, Q \notin \ell$, $P$ and $Q$ coplanar with $\ell$\label{two-coplanar}
\item $P, Q \notin \ell$, $P$ and $Q$ not coplanar with $\ell$ \label{two-not-coplanar}
\end{singularity}
\item[(\ref{three-points})] three points $P$, $Q$, $R$, not collinear
\begin{singularity}[ref=\ref*{three-points}\alph*]
\item $P, Q \in \ell$, $R \notin \ell$\label{two-on-one-off}
\item $P \in \ell$, $Q, R \notin \ell$, $Q$, $R$ coplanar with $\ell$\label{one-on-two-coplanar}
\item $P \in \ell$, $Q, R \notin \ell$, $Q$, $R$ not coplanar with $\ell$\label{one-on-two-not-coplanar}
\item $P, Q, R \notin \ell$, $P$, $Q$, $R$ and $\ell$ all coplanar\label{three-coplanar}
\item $P, Q, R \notin \ell$, $P$, $Q$ and $\ell$ coplanar, $R$ not on that plane\label{two-coplanar-one-off}
\item $P, Q, R \notin \ell$, no two coplanar with $\ell$\label{three-no-two-coplanar}
\end{singularity}
\item[(\ref{four-points})] four points $P$, $Q$, $R$, $S$, not coplanar
\begin{singularity}[ref=\ref*{four-points}\alph*]
\item $P, Q \in \ell$, $R, S \notin \ell$\label{two-on-two-off}
\item $P \in \ell$, $Q, R, S \notin \ell$, $Q$, $R$, $\ell$ coplanar\label{one-on-two-coplanar-one-off}
\item $P \in \ell$, no two of $Q$, $R$, $S$ coplanar with $\ell$\label{one-on-no-two-coplanar}
\item $P, Q, R, S \notin \ell$, $P$, $Q$, $R$ coplanar with $\ell$, but $S$ not on that plane\label{three-coplanar-one-off}
\item $P, Q, R, S \notin \ell$, $P$, $Q$ and $\ell$ coplanar, $R$, $S$ and $\ell$ coplanar\label{two-coplanar-pairs}
\item $P, Q, R, S \notin \ell$, $P$, $Q$ and $\ell$ coplanar, no other pair coplanar with $\ell$\label{two-coplanar-two-off}
\item $P, Q, R, S \notin \ell$, no two coplanar with $\ell$\label{four-no-two-coplanar}
\end{singularity}
\end{singularity}

\begin{remark}
The types correspond to orbits of the singular loci under the $\PGL(4)$ action on $\PP^3$ and the subtypes correspond to orbits under $\Stab(\ell) \subset \PGL(4)$, but this will not be explicitly important for us.
\end{remark}

\begin{definition}
For a manifold $M$ and natural number $n$, the \emph{ordered configuration space} of $n$ points on $M$ is given by 
\[\PConf_n(M) \vcentcolon= \set*{(a_1,\dots,a_n) \in M^n}{a_i \ne a_j \text{ for } i \ne j} \dispunct.\]
This space comes with a natural action of the symmetric group $\fS_n$ by permuting the coordinates and the quotient is the \emph{unordered configuration space} $\UConf_n(M)$ of $n$ points on $M$.
\end{definition}
\begin{definition}
For any $A \subseteq \UConf_n(M)$, the \emph{sign local coefficients} on $A$, denoted by $\pm \QQ$, is given by the composition
\[\pi_1(A) \to \pi_1(\UConf_n(M)) \to \fS_n \to \{\pm 1\} \subset \QQ^\times\]
thought of as a representation on $\QQ$.
Explicitly, a loop in $A$ acts on $\QQ$ by the sign of the induced permutation on the $n$ points.
\end{definition}

{The method of simplicial resolution produces for us a space $\sigma$ with a map $f \from \sigma \to \Sigma_\ell$ with the following properties:
\makeatletter
\@beginparpenalty=10000
\makeatother
\begin{enumerate}[label=(\arabic*)]
\item \label{BM-isomorphism}
The map $f_* \from \widebar H_*(\sigma) \to \widebar H_*(\Sigma_\ell)$ is an isomorphism.

\item \label{stratification}
The space $\sigma$ has a stratification 
\[\sigma = \bigcup_i F_i \dispunct, \]
where $i$ varies over all the \emph{subtypes} (not just the ones listed, but all of them).
That is, $F_i$ is a stratum corresponding to the subtype $i$.
The strata are (partially) ordered by degeneracy: $F_i$ intersects $\widebar{F_j}$ only if polynomials with singularity of subtype~$i$ can degenerate to a polynomial with singularity of subtype~$j$.

\item \label{double-bundle}
Let 
\[A_i = \{\text{singular sets $K$ of subtype~$i$}\}\]
and $K \in A_i$.
Let $L(K)$ be the linear subspace of $\Pi_\ell$ consisting of polynomials that are singular on $K$ and possibly elsewhere.
Then there are spaces $\Phi_i$ and $\Lambda(K)$ along with fiber bundles:
\[
  \begin{tikzcd}[column sep=tiny]
    L(K) \ar[hookrightarrow]{rr} & & F_i \dar[->>] \\ 
    \Lambda(K) \ar[hookrightarrow]{rr} & &\Phi_i \dar[->>] \\
 & K \rar[draw=none][description]{\in} & A_i
  \end{tikzcd}
\]

\item \label{monodromy}
The space $\Lambda(K)$ is an \emph{open cone} with vertex representing $K$ and captures the combinatorics and topology of the subsets of $K$ that can appear as singular sets of other polynomials in $\Sigma_\ell$.
The homeomorphism type of $ \Lambda(K)$ depends only on the \emph{type} of $K$ and not its subtype.
Further, $\widebar H_* ( \Lambda(K)) = 0$ unless $K$ is of type \ref{point}, \ref{two-points}, \ref{three-points}, \ref{four-points} or \ref{everything}.
For $K$ of type \ref{point}, \ref{two-points}, \ref{three-points} and \ref{four-points} respectively, i.e.\ when $K$ is a finite set of points, $ \Lambda(K)$ can be identified with the open simplex with vertex set $K$.
In particular, setting $n = \# K$, 
\[\widebar H_*(\Lambda(K)) = 0 \text{ for } * \ne n-1\dispunct,\]
and 
\[\widebar H_{n-1} (\Lambda(K)) \cong \QQ \dispunct,\] 
generated by the \emph{fundamental class}\footnote{Recall that the fundamental class of an orientable but not necessarily compact $n$-manifold $M$ without boundary is a generator of $\widebar H_{n} (M)$, and the choice of the generator corresponds to the choice of an orientation on $M$.} representing an orientation on the simplex $\Lambda(K) \cong B^{n-1}$.
Further, $A_i$ is a subset of $\UConf_n(\PP^3)$, and the monodromy on $\widebar H_* (\Lambda (K))$ is given by $\pm \QQ$ (permuting the points of $K$ changes the orientation of the simplex by the sign of the permutation).

\item \label{big-case}
For the type \ref{everything} (note that \ref{everything} has only one subtype, itself), $A_{\type{everything}}$ is singleton, the only element being $K = \PP^3$.
The only polynomial singular on $K$ is $0$, so $L(K) = \{0\}$.
Thus $F_{\type{everything}} = \Phi_{\type{everything}} =  \Lambda(\PP^3)$.
Further, the space $\Phi_{\type{everything}} = \Lambda(\PP^3)$ is the open cone over $\bigcup_{j \ne \type{everything}} \Phi_j$ for certain gluings.
\end{enumerate}
}

\begin{example}\label{parameter-space-example}
For the subtype~\ref{one-on-one-off}, a point on $\ell$ and a point not on $\ell$, we have $A_{\type{one-on-one-off}} = \ell \times \PP^3 \setminus \ell$.
For the subtype~\ref{two-not-coplanar}, two points not coplanar with $\ell$, the space $A_{\type{two-not-coplanar}}$ is an open set in $\UConf_2(\PP^3 \setminus \ell)$.
\end{example}

We refer the reader to \cite{Vassiliev99} for details of the construction and proofs for (1)--(5).
Everything we use for our computation has been summarized in these properties.
We now go through the steps of the computation before digging into the details.

By the isomorphism given by Alexander duality (\cref{alexander-duality}), we are reduced to computing $\widebar H_*(\Sigma_\ell)$.
By \ref{BM-isomorphism}, this is the same as $\widebar H_*(\sigma)$.
Let 
\[\deg(i) = 14 - \dim L(K)\]
for any $K \in A_i$.
This is monotonic on the poset described in \ref{stratification}, in the sense that if $F_i$ intersects $\widebar{F_j}$, then $\deg(i) \le \deg(j)$.
Using the filtration of $\sigma$ given by $\bigcup_{\deg(i) \le p} F_i$ there is a spectral sequence $E^r_{p,q} \implies \widebar H_{p+q} \sigma$, with the $E^1$ page given by
\begin{equation}\label{ss-terms}
E^1_{p,q} = \dSum_{\deg(i) = p} \widebar H_{p+q}(F_i) \dispunct.
\end{equation}

To compute each term, since $L(K) \to F_i \to \Phi_i$ is a complex vector bundle, we have the Thom isomorphism
\begin{equation}\label{thom-isomorphism}
\widebar H_*(F_i) = \widebar H_{* -2 \dim_\CC L(K)} (\Phi_i) \dispunct.
\end{equation}
For the right-hand side, if $ \Lambda(K)$ is acyclic then so must be $\Phi_i$, so this automatically vanishes unless $i$ is a subtype of \ref{point}, \ref{two-points}, \ref{three-points}, \ref{four-points}, or \ref{everything}.

For the (sub)type \ref{everything}, from \ref{big-case} we have that $\Phi_{\type{everything}} =  C Z$, the open cone on $Z$, where 
\[Z = \bigcup_{i \ne \type{everything}} \Phi_i \dispunct.\]
So we get a spectral sequence $e^r_{p,q} \implies H_{p+q} (Z)$ with
\[e^1_{p,q} = \dSum_{\substack{\deg(j) = p,\\ j \ne \type{everything}}} \widebar H_{p+q}(\Phi_j) \dispunct.\]
But then we also have
\[\widebar H_*( CZ) = H_*(CZ , Z) = \widetilde H_{*-1}(Z) \dispunct.\]

For all the other $i$, the set $K$ is finite, of say $n$ points ($1 \le n \le 4$).
Then as described in \ref{monodromy}, $\widebar H_*( \Lambda(K))$ is concentrated in degree $n-1$, so 
\begin{equation}\label{local-coefficient-isomorphism}
\widebar H_*(\Phi_i) = \widebar H_{* - n + 1}(A_i; \pm \QQ) = H^{2\dim_\CC A_i + n - 1 - *}(A_i; \pm \QQ)   \dispunct,
\end{equation}
where the latter isomorphism is by (twisted) Poincaré duality, since the $A_i$ are complex manifolds.
So the computation eventually boils down to computing $H^*(A_i; \pm \QQ)$ for these $i$ (see~\cref{vanishing-cohomology,shape-of-page-1}), bookkeeping, and then relatively standard arguments involving spectral sequences following \cite{Vassiliev99} (see~\cref{cohomology-of-X-l}).

Before we start on the detailed casework, it is worth describing representatives of (multiplicative) generators of $H^*(\widetilde X_\ell)$.
From \cref{surjection}, we know that $H^*(\widetilde X_\ell)$ is generated in degrees $1$ and $3$.
Tracing through all of the algebra above and the degeneration at $E^1_{p,q}$ as described in \cref{cohomology-of-X-l}, we have isomorphisms:
\begin{align*}
H^1(\widetilde X_\ell) \cong \widebar H_{30}(\Sigma_\ell) &\cong H^0(A_{\type{one-on}}) \dsum H^0(A_{\type{one-off}})\\
H^3(\widetilde X_\ell) \cong \widebar H_{28}(\Sigma_\ell) & \cong H^2(A_{\type{one-on}}) \dsum H^2(A_{\type{one-off}})
\end{align*}
Note that $A_{\type{one-on}} = \ell$ and $A_{\type{one-off}} = \PP^3 \setminus \ell$ (which by \cref{hyperspace-complement} deformation retracts to $\ell^\perp$).
For a more geometric description, consistent with the proof of \cref{surjection}, we can find representative subspaces of $\Sigma_\ell$, after fixing $P \in \ell$, and $P' \in \ell^\perp$.
Again, tracing through the chain of isomorphisms above, the subspaces corresponding to (i)~$H^0(A_{\type{one-on}})$, (ii)~$H^2(A_{\type{one-on}})$, (iii)~$H^0(A_{\type{one-off}})$ and (iv)~$H^2(A_{\type{one-off}})$ are (i)~polynomials singular at~$P$, (ii)~polynomials singular at some point of~$\ell$, (iii)~polynomials singular at~$P'$ and (iv)~polynomials singular at some point of~$\ell'$.

\begin{remark}
In this entire computation, we could keep track of the mixed Hodge structures throughout, as in \cite{Tommasi05,Tommasi14} (see~also~\cite{Gorinov05}), but this ends up being unnecessary for our purposes.
This information can in any case be recovered a posteriori given \cref{surjection,cohomology-of-X-l}, since the orbit map is algebraic.
\end{remark}

\subsection{General results on configuration spaces of projective space}

We now state the results of some general computations that we will use in the case work, since the arguments needed are fairly independent.

\begin{lemma}\label{hyperspace-complement}
Let $H \cong \PP^k$ be a $k$-dimensional linear subspace of $\PP^n$ for some $0 \le k \le n$ and let $H^\perp$ be the (projectivized) orthogonal complement of $H$.
Then $\PP^n \setminus H$ deformation retracts to $H^\perp \cong \PP^{n-k-1}$.
\end{lemma}
\begin{proof}
Let $\PP^n = \{[x_0: \dots : x_n]\}$ and without loss of generality, $H = \{[x_0: \dots : x_k]\}$.
Then 
\[(t,[x_0: \dots: x_k: x_{k+1} : \dots : x_n]) \mapsto [tx_0: tx_1: \dots : tx_k: x_{k+1} : \dots : x_n]\]
is an explicit deformation retract.
\end{proof}

\begin{lemma}\label{totaro-computations}
\[H^*(\UConf_2(\CC); \pm \QQ) \cong H^*(\UConf_2(\CC^2),\pm \QQ) \cong H^*(\UConf_4(\CC^2); \pm \QQ) = 0 \dispunct.\]
\[H^*(\UConf_4(\PP^2 \setminus \{\bullet\}); \pm \QQ) \cong H^*(\UConf_4(\PP^3 \setminus \PP^1); \pm \QQ) = 0\]
\[H^*(\UConf_2(\PP^1); \pm \QQ) \cong H^*(\UConf_2(\PP^3 \setminus \PP^1); \pm \QQ) \cong \begin{cases*}
\QQ & if $* = 2$\\
0 &otherwise.\end{cases*}\]
\end{lemma}
\begin{proof}
For $\UConf_2(\CC)$ or $\UConf_2(\CC^2)$, we can use that $\PConf_2(\RR^{2n}) \simeq S^{2n-1}$, and the $\fS_2$ action is by the antipodal map, which is degree $1$ and hence by transfer $H^*(\UConf_2(\RR^{2n}); \pm \QQ) = 0$.

For all the other spaces of the form $\UConf_n(Z)$, \cite{Totaro96} provides spectral sequences that converge to $\PConf_n(Z)$ as an $\fS_n$ representation.
The computation of each of these is straightforward from \cite[Theorem 1]{Totaro96}.
The conclusion again follows from transfer, since $H^*(\UConf_n(Z); \pm \QQ)$ is the $\pm \QQ$ summand of $H^*(\PConf_n(Z); \QQ)$ as a $\fS_n$ representation.

For $H^*(\UConf_2(\PP^1); \pm \QQ)$ we can also use \cite[Lemma~2B]{Vassiliev99}.
\end{proof}

\subsection{Case work}

This section contains the details of the arguments to compute the various $H^*(A_i; \pm \QQ)$.
The main idea is decomposing these spaces as fiber bundles, where both the fiber and base are simpler.
In many instances the bases are $A_j$ for some lower $j$, and the computation is `inductive' or recursive.
First we establish the cases where the answer is $0$, the recursive nature of the argument makes some of the cases relatively easy.
The cases that are exceptions in the proposition below are treated in \cref{shape-of-page-1}.

\begin{proposition}\label{vanishing-cohomology}
Suppose that $i$ is a subtype of \ref{point}, \ref{two-points}, \ref{three-points} or \ref{four-points}.
Then $H^*(A_i; \pm \QQ) = 0$ unless $i$ is one of \ref{one-on}, \ref{one-off}, \ref{two-on}, \ref{one-on-one-off}, \ref{two-not-coplanar}, \ref{two-on-one-off}, \ref{one-on-two-not-coplanar}, \ref{two-on-two-off}, \ref{everything}.
\end{proposition}

\begin{proof}
We need to show that $H^*(A_i; \pm \QQ) = 0$ when $i$ is one of 
\ref{two-coplanar},
\ref{one-on-two-coplanar},
\ref{three-coplanar},
\ref{two-coplanar-one-off},
\ref{three-no-two-coplanar},
\ref{one-on-two-coplanar-one-off},
\ref{one-on-no-two-coplanar},
\ref{three-coplanar-one-off},
\ref{two-coplanar-pairs},
\ref{two-coplanar-two-off},
\ref{four-no-two-coplanar}.
Let's deal with each in turn.

\begin{description}[wide,itemindent=*]

\item[\ref{two-coplanar}, $P, Q \notin \ell$, but $P$, $Q$ and $\ell$ coplanar]
Mapping $\{P,Q\} \mapsto H = \gen{P,Q,\ell}$, the projective span of $P,Q,\ell$, i.e.\ the plane containing $P$, $Q$ and $\ell$, we get a map from $A_{\type{two-coplanar}}$ to the space of planes in $\PP^3$ containing $\ell$, which is a $\PP^1 \cong \ell^\vee \subset (\PP^3)^{\vee}$.
This is a fiber bundle
\[\begin{tikzcd}
\UConf_2(H \setminus \ell) \arrow[r,hook] & A_{\type{two-coplanar}} \arrow[d,->>] \\ & \PP^1
\end{tikzcd}\]
and the local coefficients $\pm \QQ$ restrict to the fiber to the sign local coefficient on $\UConf_2(H \setminus \ell) \cong \UConf_2(\CC^2)$.
But $H^*(\UConf_2(\CC^2),\pm \QQ) = 0$ from \cref{totaro-computations}, so we are done.

\item[\ref{one-on-two-coplanar}, $P \in \ell$, $Q, R \notin \ell$, but $Q$, $R$ and $\ell$ coplanar]
Here, even though $P$, $Q$ and $R$ are a priori unordered, we can't (continuously) swap $R$ with one of $P$ and $Q$.
So there is a well-defined map $\{P,Q,R\} \mapsto \{Q,R\}$, and we get a fiber bundle:
\[\begin{tikzcd}
\CC \cong \ell \setminus \gen{Q,R} \rar[hookrightarrow] & A_{\type{one-on-two-coplanar}} \dar[->>] \\ & A_{\type{two-coplanar}}
\end{tikzcd}\]
The local coefficients $\pm \QQ$ on the total space pull-back from $\pm \QQ$ on base (that is, the map $\pi_1(A_{\type{one-on-two-coplanar}}) \to \{\pm 1\}$ factors through $\pi_1(A_{\type{two-coplanar}})$.
But as we just showed, $H^*(A_{\type{two-coplanar}}; \pm \QQ) = 0$, so we are done.

\item[\ref{three-coplanar}, $P, Q, R \notin \ell$, but $P$, $Q$, $R$ and $\ell$ coplanar]
Mapping $\{P,Q,R\} \mapsto H = \gen{P,Q,R,\ell}$, we get a fiber bundle:
\[\begin{tikzcd}
F \rar[hookrightarrow] & A_{\type{three-coplanar}} \dar[->>] \\ & \PP^1
\end{tikzcd}\]
The fiber is the space of three (unordered points) non-collinear points on $H \setminus \ell \cong \CC^2$, and the local coefficients $\pm \QQ$ restrict to the local coefficients $\pm \QQ$ on $F \subset \UConf_3(\CC^2)$.
Since $\pi_1(F) \to \{\pm 1\}$ factors through $\fS_3$, we can go to the associated $\fS_3$ cover $\widetilde F \subset \PConf_3(\CC^2)$, and then by transfer, $H^*(F; \pm \QQ)$ is the summand of $H^*(\widetilde F; \QQ)$ where $\fS_3$ acts by the sign representation.
But $\widetilde F = \{(P,Q,R)\}$ can be broken up as fiber bundles (see~\cref{fig:three-coplanar}):
\[\begin{tikzcd}
\CC^2 \setminus \gen{P,Q} \rar[hookrightarrow] & \widetilde F = \{(P,Q,R)\} \dar[->>] \\ \CC^2 \setminus \{P\} \rar[hookrightarrow] & \{(P,Q)\} \dar[->>] \\ & \CC^2 = \{P\}
\end{tikzcd}\]
So $H^*(\widetilde F; \QQ) = H^*(S^1 \times S^3; \QQ)$, but more importantly for us, we show that the $\fS_3$ action on $H^*(\widetilde F; \QQ)$ is trivial, which implies $H^*(F; \pm \QQ) = 0$ as needed.
It is enough to check that on each generator (which has to come from one of the fibers or the base), the transposition acts trivially (since transpositions generate $\fS_3$).
The transposition $(QR)$ acts trivially on $\CC^2 \setminus \{P\}$, and hence trivially on the generator of $H^3(\CC^2 \setminus \{P\}) \cong H^3(\widetilde F)$ (the latter description is independent of the choice of transposition).
Similarly the transposition $(PQ)$ acts by $-1$ on $\CC^2 \setminus \gen{P,Q}$, but this is degree $1$ on an even-dimensional vector space (and odd-dimensional sphere), so acts trivially on the generator of $H^1(\widetilde F)$.

\begin{figure}
\centering
\includegraphics{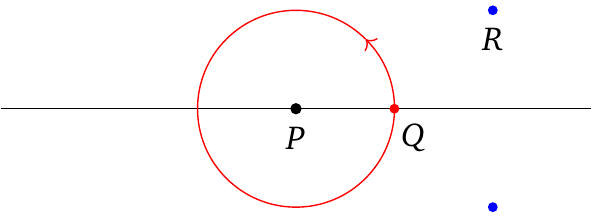}
\caption{The (real) fibering of $\widetilde F$ in the case \ref{three-coplanar}.
The spaces $\CC^2 \setminus \{P\} \simeq S^3$ and $\CC^2 \setminus \gen{P,Q} \simeq S^1$ in the complex fiber respectively correspond to the pictured $S^1$ and $S^0$ in the real points.}
\label{fig:three-coplanar}
\end{figure}

\item[\ref{two-coplanar-one-off}, $P, Q, R \notin \ell$, $P$, $Q$ and $\ell$ coplanar, but $R$ not on that plane]
Mapping $\{P,Q,R\} \mapsto \{P,Q\}$, we get a fiber bundle
\[\begin{tikzcd}
\CC^3 \cong \PP^3 \setminus \gen{P,Q,\ell} \rar[hookrightarrow] & A_{\type{two-coplanar-one-off}} \dar[->>] \\ & A_{\type{two-coplanar}}
\end{tikzcd}\]
and we are again done, similar to the case \ref{one-on-two-coplanar} above. 

\item[\ref{three-no-two-coplanar}, $P, Q, R \notin \ell$, no two coplanar with $\ell$]

In this case, we go to the $\fS_3$ cover $\widetilde A$ of $A_{\type{three-no-two-coplanar}}$, so that similar to above, $H^*(A_{\type{three-no-two-coplanar}}; \pm \QQ)$ is the sign-representation summand of $H^*(\widetilde A; \QQ)$.
Then $\widetilde A$ can be broken up by fiber bundles (see~\cref{fig:three-no-two-coplanar}):
\[\begin{tikzcd}
(\CC \setminus 0) \times (\CC^2 \setminus 0) \cong  \PP^3 \setminus (\gen{P,\ell} \cup \gen{Q,\ell} \cup \gen{P,Q}) \rar[hookrightarrow] & \widetilde A \dar[->>] \\
 \CC^3 \cong  \PP^3\setminus \gen{P,\ell} \rar[hookrightarrow] & \{(P,Q)\}\dar[->>]\\
&  \{P\} \rar[equal] & \PP^3 \setminus \ell \simeq \PP^1
\end{tikzcd}\]
The $\fS_3$ action on $H^*(\widetilde A)$ is again trivial by arguments similar to above.

\begin{figure}
\centering
\includegraphics{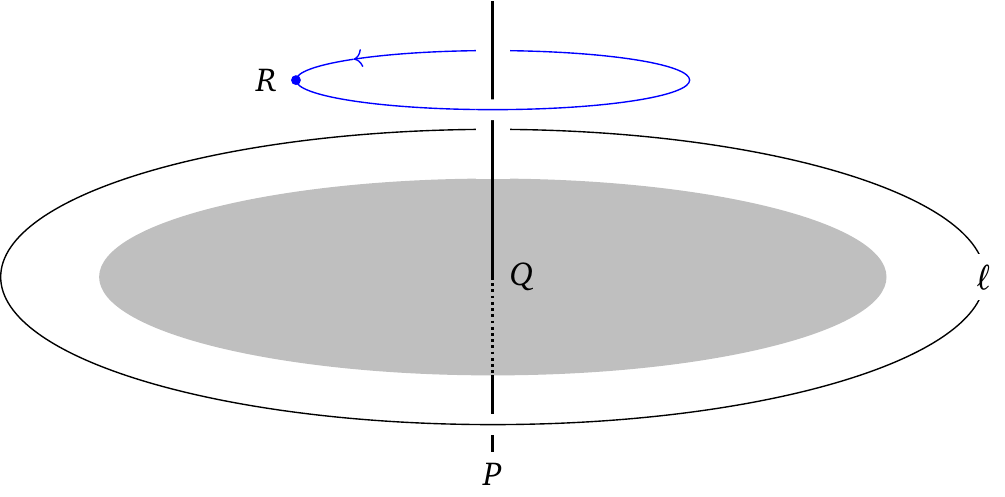}
\caption{The (real) fiber for $R$ in the cover $\widetilde A$ of $A_{\type{three-no-two-coplanar}}$.
The complex fiber is homeomorphic to $(\CC \setminus 0) \times (\CC^2 \setminus 0)$.
The point $P$ and the line $\ell$ are at infinity, and the loop pictured corresponds to the generator in $H^3(\CC^2 \setminus 0)$.}
\label{fig:three-no-two-coplanar}
\end{figure}

\item[\ref{one-on-two-coplanar-one-off}, $P \in \ell$, $Q,R,S \notin \ell$, $Q$ and $R$ coplanar with $\ell$, $S$ not on that plane]
Similar to above, mapping $\{P,Q,R,S\} \mapsto \{Q,R,S\}$ we get a fiber bundle:
\[\begin{tikzcd}
\ell \setminus \gen{Q, R, S} \rar[hookrightarrow] & A_{\type{one-on-two-coplanar-one-off}} \dar[->>] \\
& A_{\type{two-coplanar-one-off}}
\end{tikzcd}\]
We are done by previous arguments.

\item[\ref{one-on-no-two-coplanar}, $P \in \ell$, $Q, R, S \notin \ell$, no two of $Q$, $R$ and $S$ coplanar with $\ell$]
Mapping $\{P,Q,R,S\} \mapsto \{Q,R,S\}$ we get a fiber bundle:
\[\begin{tikzcd}
\ell \setminus \gen{Q, R, S} \rar[hookrightarrow] & A_{\type{one-on-no-two-coplanar}} \dar[->>] \\
& A_{\type{three-no-two-coplanar}}
\end{tikzcd}\]
We are done by previous arguments.

\item[\ref{three-coplanar-one-off}, $P,Q,R,S \notin \ell$, $P$, $Q$, $R$ coplanar with $\ell$, $S$ not on that plane]
Mapping $\{P,Q,R,S\} \mapsto \{P,Q,R\}$:
\[\begin{tikzcd}
\CC^3 \cong \PP^3 \setminus \gen{P, Q, R, \ell} \rar[hookrightarrow] & A_{\type{three-coplanar-one-off}} \dar[->>] \\
& A_{\type{three-coplanar}}
\end{tikzcd}\]
We are done by previous arguments.

\item[\ref{two-coplanar-pairs}, $P,Q,R,S \notin \ell$, $P$ and $Q$ coplanar with $\ell$, $R$ and $S$ coplanar with $\ell$]
We can map $\{P,Q,R,S\} \mapsto \{\gen{P, Q}, \gen{R, S}\}$, the two lines through $PQ$ and $RS$ and get a map $A_{\type{two-coplanar-pairs}} \to B$, where $B$ is the set of unordered pairs of lines in $\PP^3$ that both intersect $\ell$, but so that the three lines are not coplanar (in particular the pair of lines do not themselves intersect).
This is a fiber bundle:
\[\begin{tikzcd}
 \UConf_2(L_1 \setminus \ell) \times \UConf_2(L_2 \setminus \ell) \rar[hookrightarrow] & A_{\type{two-coplanar-pairs}} \dar[->>] \\
 & B
\end{tikzcd}\]
Since $L_i \setminus \ell \cong \CC$, and $H^*(\UConf_2(\CC),\pm \QQ) = 0$, $H^*(A_{\type{two-coplanar-pairs}}; \pm \QQ) = 0$.

\item[\ref{two-coplanar-two-off}, $P,Q,R,S \notin \ell$, $P$ and $Q$ coplanar with $\ell$, no other pair coplanar with $\ell$]
Forthis case mapping $\{P,Q,R,S\} \mapsto \{P,Q\}$, we get a fiber bundle:
\[\begin{tikzcd}
\{\{R, S\}\} \rar[hookrightarrow] & A_{\type{two-coplanar-two-off}} \dar[->>] \\
& A_{\type{two-coplanar}}
\end{tikzcd}\]
Here the local coefficients $\pm \QQ$ on $A_{\type{two-coplanar-two-off}}$ is induced by $\pm \QQ$ on $A_{\type{two-coplanar}}$ and $\pm \QQ$ on the fibers $\{\{R,S\}\} \subset \UConf_2(\PP^3)$.
We are done by previous arguments.

\item[\ref{four-no-two-coplanar}, $P, Q, R, S \notin \ell$, no two coplanar with $\ell$]
By an argument analogous to the case of \ref{three-no-two-coplanar}, $A_{\type{four-no-two-coplanar}}$ has an $\fS_4$ cover by ordering the four points that breaks up as a fiber bundle over the $\fS_3$ cover of $A_{\type{three-no-two-coplanar}}$.
The sign representation doesn't occur in the cohomology of this cover, so we are done.

Alternatively, one can note that $\UConf_4(\PP^3 \setminus \ell)$ has the following stratification:
\[\UConf_4(\PP^3 \setminus \ell) = \{\text{four coplanar points in $\PP^3 \setminus \ell$}\}  \sqcup A_{\type{three-coplanar-one-off}} \sqcup A_{\type{two-coplanar-pairs}} \sqcup A_{\type{two-coplanar-two-off}} \sqcup A_{\type{four-no-two-coplanar}} \dispunct. \]
The first term can be further stratified into two sets:
\[\{\text{four coplanar points in $\PP^3 \setminus \ell$}\} = Y_0 \sqcup Y_1 \dispunct,\]
where
\[Y_0 = \{\text{four points in $H \setminus \ell$ for $H$ some plane in $\PP^3$ containing $\ell$}\} \dispunct,\]
and 
\[Y_1 = \{\text{four points in $H \setminus \ell$ for $H$ some plane in $\PP^3$ \emph{not} containing $\ell$}\} \dispunct.\]
In either case, mapping to the plane $H$ gives us two fiber bundles 
\[\begin{tikzcd}
 \UConf_4(\CC^2) \cong \UConf_4(H \setminus \ell) \rar[hookrightarrow] & Z_0 \dar[->>] \\
 & \{H \supset \ell\}
\end{tikzcd} \qquad \begin{tikzcd}
 \UConf_4(\PP^2 \setminus \{\bullet\}) \cong \UConf_4(H \setminus \ell) \rar[hookrightarrow] & Y_1 \dar[->>] \\
 & \{H \not\supset \ell\}
\end{tikzcd}\]
Using \cref{totaro-computations} and previous arguments, 
\[H^*(\UConf_4(\PP^3 \setminus \PP^1); \pm \QQ) \cong H^*(Y_0; \pm \QQ) \cong H^*(Y_1; \pm \QQ) = 0 \dispunct.\]
Since we've already shown that $A_{\type{three-coplanar-one-off}}$, $A_{\type{two-coplanar-pairs}}$ and $A_{\type{two-coplanar-two-off}}$ are $\pm \QQ$-acyclic, $A_{\type{four-no-two-coplanar}}$ must be as well.
\end{description}
\end{proof}

Recall that by \ref{stratification} we have spectral sequences $E^r_{p,q} \implies \widebar H_{p+q} (\sigma)$ and $e^r_{p,q}$ that let us compute $\widebar H_*(F_{\type{everything}}) = \widetilde H_{*-1}(Z)$, where 
\[Z = \bigcup_{i \ne \type{everything}} \Phi_i \dispunct.\]

\begin{proposition}\label{shape-of-page-1}
The spectral sequence $E^r_{p,q} \implies \widebar H_{p+q} (\sigma)$ has the page $E^1_{p,q}$ as in \cref{big-SS}.
The spectral sequence $e^r_{p,q} \implies H_{p+q} (Z)$ has the page $e^1_{p,q}$ as in \cref{small-SS}.
\end{proposition}

\begin{figure}
\centering
\includegraphics{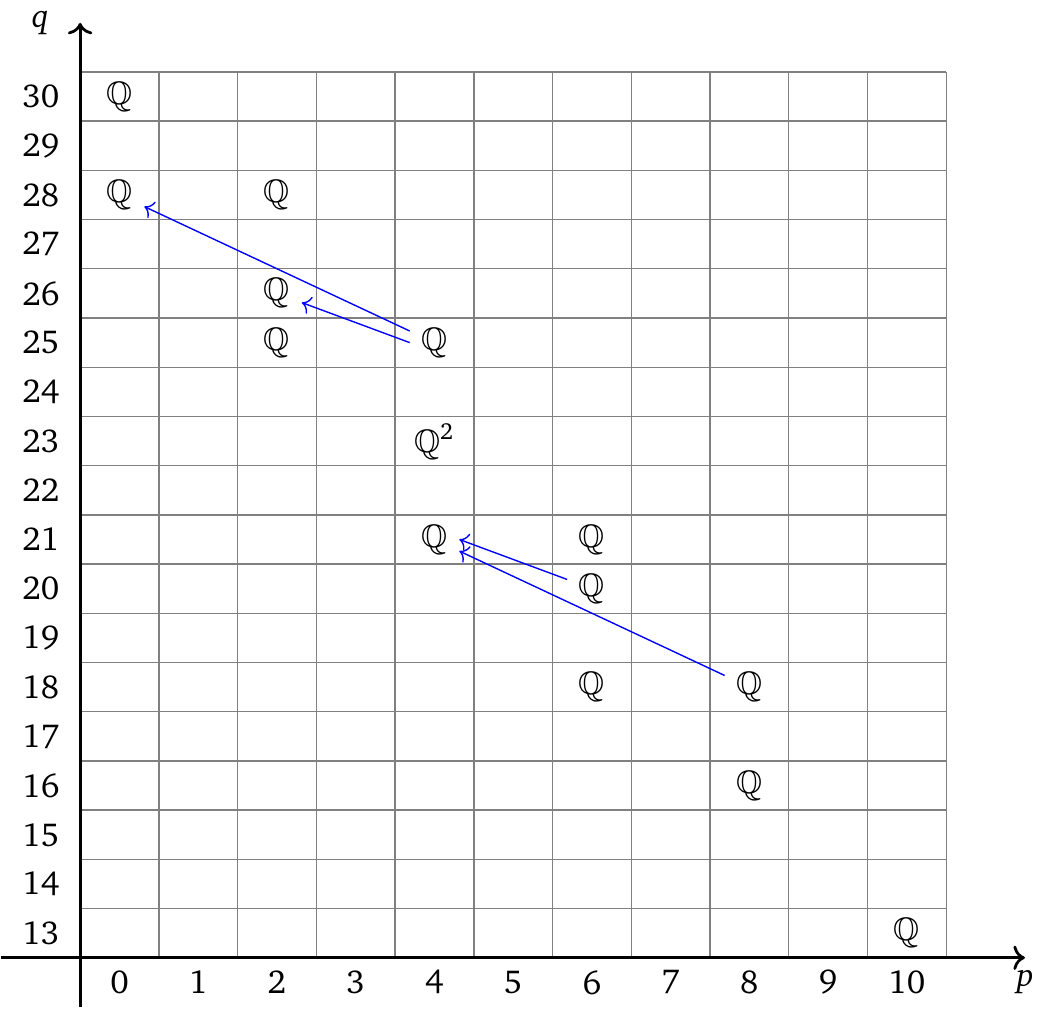}
\caption{Spectral sequence page $E^1_{p,q}$ for $\widebar H_{p+q} (\sigma)$ (with $0$s omitted) and all potentially non-zero differentials in subsequent pages}
\label{big-SS}
\end{figure}

\begin{figure}
\centering
\includegraphics{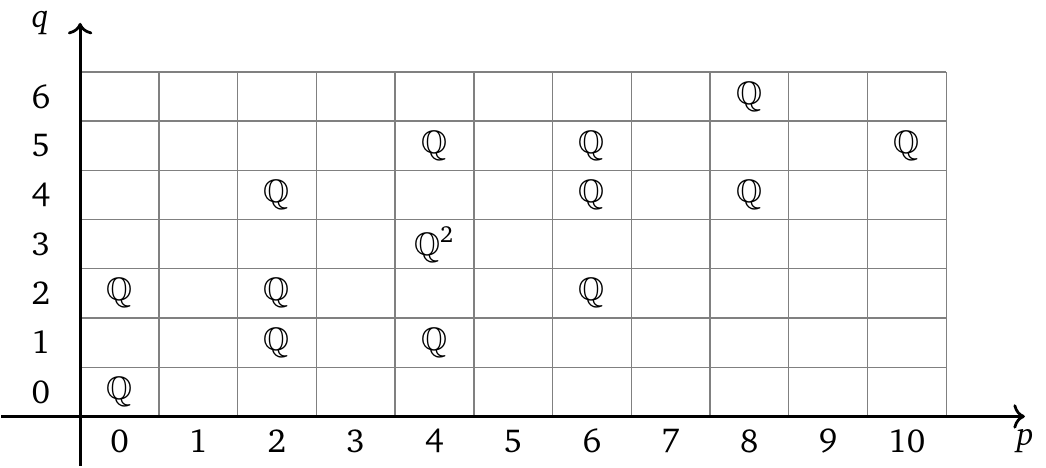}
\caption{Spectral sequence page $e^1_{p,q}$ for $H_{p+q}(Z)$ (with $0$s omitted)}
\label{small-SS}
\end{figure}

\begin{proof}
Recall that by construction, the terms of $E^1$ and $e^1$ are related by Thom isomorphisms:
\[E^1_{p,q + 2(14-p)} \cong e^1_{p,q}\]
except for $p = 14$, where $e^1_{14,*} \equiv 0$.
So we first go through more case work to establish columns $p \ne 14$.

By \cref{ss-terms,thom-isomorphism,local-coefficient-isomorphism} and careful bookkeeping, it is enough to find $H^*(A_i; \pm \QQ)$ along with the numbers $\dim(A_i) = \dim_\CC(A_i)$ and $\dim(L(K)) = \dim_\CC(L(K))$ for $K \in A_i$, for the subtypes $i$ of \ref{point}, \ref{two-points}, \ref{three-points} \ref{four-points} (see~\cref{numbers-table} for the relevant numerics).
Further, there are only eight subtypes remaining --- the exceptions from \cref{vanishing-cohomology}.

\begin{table}
\centering
\begin{tabular}{crrrrrrrrr}
\toprule
$i$ & \ref{one-on} & \ref{one-off} & \ref{two-on} & \ref{one-on-one-off} & \ref{two-not-coplanar} & \ref{two-on-one-off} & \ref{one-on-two-not-coplanar} & \ref{two-on-two-off} & \ref{everything}\\
\midrule 
$\dim A_i$ & $1$ & $3$ & $2$ & $4$ & $6$ & $5$ & $7$ & $8$ & $0$ \\
$\dim L(K)$ & $14$ & $12$ & $12$ & $10$ & $8$ & $8$ & $6$ & $4$ & $0$ \\
\bottomrule
\end{tabular}
\caption{$\dim A_i$ and $\dim L(K)$ for $K \in A_i$ for each subtype $i$ excepted in \cref{vanishing-cohomology}.} \label{numbers-table}
\end{table}

\begin{description}[wide]
\item[\ref{one-on}, $P \in \ell$]
$A_{\type{one-on}} = \ell \cong \PP^1$, since there is only one point, the coefficients $\pm \QQ$ are trivial, so
\[H^*(A_{\type{one-on}}; \pm \QQ) = H^*(\PP^1) = \begin{cases*} \QQ & $* = 0,2$\\ 0 & otherwise.\end{cases*}\]
This contributes to $E^1_{0,28} \cong e^1_{0,0}$ and $E^1_{0,30} \cong e^1_{0,2}$ since $\dim(A_{\type{one-on}}) =  1$ and $\dim(L(K)) = 14$.

\item[\ref{one-off}, $P \notin \ell$]
$A_{\type{one-off}} = \PP^3 - \ell \simeq \PP^1$.
Again, the coefficients are trivial, so
\[H^*(A_{\type{one-off}}; \pm \QQ) = H^*(\PP^1) = \begin{cases*} \QQ & $* = 0,2$\\ 0 & otherwise.\end{cases*}\]
This contributes to $E^1_{2,26} \cong e^1_{2,2}$ and $E^1_{2,28} \cong e^1_{2,4}$ since $\dim(A_{\type{one-off}}) =  3$ and $\dim(L(K)) = 12$.

\item[\ref{two-on}, $P, Q \in \ell$]
$A_{\type{two-on}} = \UConf_2(\ell) \cong \UConf_2(\PP^1)$.
By \cref{totaro-computations},
\[H^*(A_{\type{two-on}}; \pm \QQ) = \begin{cases*} \QQ & $* = 2$\\ 0 & otherwise.\end{cases*}\]
This contributes to $E^1_{2,25} \cong e^1_{2,1}$ since $\dim(A_{\type{two-on}}) =  2$ and $\dim(L(K)) = 12$.

\item[\ref{one-on-one-off}, $P\in \ell$, $Q \notin \ell$]
$A_{\type{one-on-one-off}} \cong \ell \times (\PP^3 \setminus \ell) \simeq \PP^1 \times \PP^1$ and the coefficients are trivial.
Hence,
\[H^*(A_{\type{one-on-one-off}}; \pm \QQ) \cong H^*(\PP^1 \times \PP^1) = \begin{cases*} \QQ & $* = 0, 4$\\ \QQ^2 & $* = 2$ \\ 0 & otherwise.\end{cases*}\]
This contributes to $E^1_{4,21} \cong e^1_{4,1}$, $E^1_{4,23} \cong e^1_{4,3}$ and $E^1_{4,25} \cong e^1_{4,5}$ since $\dim(A_{\type{two-on}}) =  4$ and $\dim(L(K)) = 10$.

\item[\ref{two-not-coplanar}, $P, Q \notin \ell$, $P$ and $Q$ not coplanar with $\ell$]
$A_{\type{two-not-coplanar}} = \UConf_2(\PP^3 \setminus \ell) \setminus A_{\type{two-coplanar}})$.
\Cref{vanishing-cohomology} shows that $H^*(A_{\type{two-coplanar}}; \pm \QQ) = 0$, so from the Gysin sequence, and by \cref{totaro-computations},
\[H^*(A_{\type{two-not-coplanar}}, \pm \QQ) \cong H^*(\UConf_2(\PP^3 \setminus \PP^1),\pm \QQ) = \begin{cases*} \QQ & $* = 0,2$\\ 0 & otherwise.\end{cases*}\]
This contributes to $E^1_{6,21} \cong e^1_{6,5}$ since $\dim(A_{\type{two-on}}) =  6$ and $\dim(L(K)) = 8$.

\item[\ref{two-on-one-off}, $P, Q \in \ell$, $R \notin \ell$]
$A_{\type{two-on-one-off}} \cong \UConf_2(\ell) \times (\PP^3 \setminus \ell)$, and the local coefficients restrict to trivial coefficients on the second factor $\PP^3 \setminus \ell \simeq \PP^1$.
Thus,
\[H^*(A_{\type{two-on-one-off}}; \pm \QQ) \cong \dsum_{a+b = *} H^a(\UConf_2(\PP^1); \pm \QQ) \tensor H^b(\PP^1) = \begin{cases*} \QQ & $* = 2,4$\\ 0 & otherwise.\end{cases*}\]
This contributes to $E^1_{6,18} \cong e^1_{6,2}$ and $E^1_{6,20} \cong e^1_{6,4}$ since $\dim(A_{\type{two-on-one-off}}) = 5$ and $\dim(L(K)) = 8$.

\item[\ref{one-on-two-not-coplanar}, $P \in \ell$, $Q, R \notin \ell$, $Q$ and $R$ not coplanar with $\ell$]
Since the line $\gen{Q,R}$ doesn't intersect $\ell$, $P$ can be any point on $\ell$ for any choice of $Q$ and $R$.
Thus $A_{\type{one-on-two-not-coplanar}} = \ell \times A_{\type{two-not-coplanar}}$ and the local coefficients are trivial on the first factor ($\ell$ is anyway simply connected).
Hence
\[H^*(A_{\type{one-on-two-not-coplanar}}; \pm \QQ) \cong \dsum_{a+b = *} H^a(\PP^1) \tensor H^b(A_{\type{two-not-coplanar}}; \pm \QQ) = \begin{cases*} \QQ & $* = 2,4$\\ 0 & otherwise.\end{cases*}\]
This contributes to $E^1_{8,16} \cong e^1_{8,4}$ and $E^1_{8,18} \cong e^1_{8,6}$ since $\dim(A_{\type{two-on-one-off}}) = 7$ and $\dim(L(K)) = 6$.

\item[\ref{two-on-two-off}, $P, Q \in \ell$, $R, S \notin \ell$]
By definition of \ref{four-points}, the four points cannot be coplanar.
This is equivalent to $R$ and $S$ not being coplanar with $\ell$.
If $\rho \from \PP^3 \setminus \ell \to \ell^\perp$ is the projection, then this is further equivalent to $\rho(R) \ne \rho(S)$.
Note that $\rho^{-1}(T) = \gen{T, \ell} \setminus \ell \cong \CC^2$.
Thus, mapping $\{P,Q,R,S\} \mapsto (\{P,Q\}, \{\rho(R), \rho(S)\})$, we get a bundle:
\[\begin{tikzcd}
 \CC^4 \rar[hookrightarrow] & A_{\type{two-on-two-off}} \dar[->>] \\
 & \UConf_2(\ell) \times \UConf_2(\ell^\perp) 
\end{tikzcd}\]
This implies, using \cref{totaro-computations},
\[H^*(A_{\type{two-on-two-off}}; \pm \QQ) \cong \dsum_{a+b = *} H^a(\UConf_2(\PP^1); \pm \QQ) \tensor H^a(\UConf_2(\PP^1); \pm \QQ) = \begin{cases*} \QQ & $* = 4$\\ 0 & otherwise.\end{cases*}\]
This contributes to $E^1_{10,13} \cong e^1_{10,5}$ since $\dim(A_{\type{two-on-two-off}}) = 8$ and $\dim(L(K)) = 4$.
\end{description}

Thus we've computed the pages $e^1_{p,q}$ and $E^1_{p,q}$ except the $p = 14$ column of the latter.
For \ref{everything}, $L(K) = 0$, so $E^1_{14,q} \cong \widebar H_{14+q}(\Phi_{\type{everything}})$.
Now, if any term with $1 \le d = p+q \le 14$ remains non-zero in $e^\infty_{p,q}$, then it would appear as $\widebar H_{d+1}(\Phi_{\type{everything}})$ and hence as a term $E^1_{14, d-13}$, which cannot interact with any of the other terms, by the shapes of the other columns, which we have already determined.
That means $0 \ne \widebar H_{d+1} (\sigma) \cong \widetilde H^{31-d} (\widetilde X_\ell)$, which is a contradiction with $\widetilde X_\ell$ being a $16$-dimensional Stein manifold, as in \cref{andreotti--frankel}.
This implies, given the shape of $e^1_{p,q}$, that $\widebar H_*(\Phi_{\type{everything}}) \equiv 0$, so we have also verified $E^1_{14,*}$.
\end{proof}

\begin{proposition}\label{cohomology-of-X-l}
The spectral sequence $E^r_{p,q}$ degenerates at $r = 1$ and hence the rational cohomology of $\widetilde X_\ell$ is given by 
\[H^*(\widetilde X_\ell; \QQ) \cong \begin{cases*}
\QQ & if $* = 0, 2, 6, 8$\\
\QQ^2 & if $* = 1, 3, 5, 7$\\
\QQ^4 & if $* = 4$\\
0 & otherwise.
\end{cases*}\]
\end{proposition}
\begin{proof}
Recall that $E^r_{p,q} \implies \widebar H_{p+q} (\sigma) \cong \widetilde H^{31 - p - q}(\sigma)$.
The page $E^1_{p,q}$ is quite sparse to begin with, the only potentially non-zero differentials (on any page) are shown in \cref{big-SS}.
By \cref{hypersurface-complement}, since $\widetilde X_\ell = \Pi_l \setminus \cV(\Delta_\ell)$, we must have 
\[P_\QQ(\widetilde X_\ell, t) = P_\QQ(\CC^\times, t)P_\QQ(\widetilde M_\ell, t) = (1+t)P_\QQ(\widetilde M_\ell, t) \dispunct,\]
where $P_\QQ(\_, t)$ denotes the (rational) Poincaré polynomial.
This shows that $H^2(\widetilde X_\ell) \cong \widebar H_{29}(\sigma)$ and $H^6(\widetilde X_\ell) \cong \widebar H_{25}(\sigma)$ cannot be $0$, which means all those differentials must vanish.
So $E^\infty_{p,q} \cong E^1_{p,q}$ and there are no extension problems with rational coefficients.
\end{proof}

%
%
%
%
%
%

\printbibliography

\end{document}